\documentclass{amsart}
\usepackage{amsmath, amssymb,amsthm,bm,mathrsfs}
\usepackage[utf8]{inputenc}
\usepackage{enumerate}
\usepackage{graphicx}
\usepackage{tikz}
\usetikzlibrary{arrows.meta,patterns}

\usepackage{xcolor}

\newtheorem{theorem}{Theorem}
\newtheorem{lemma}{Lemma}
\newtheorem{corollary}{Corollary}
\newtheorem{conjecture}{Conjecture}

\theoremstyle{definition}

\theoremstyle{remark}
\newtheorem{remark}{Remark}

\title[$\ell^p$-improving for discrete curves]{Some subcritical estimates for the $\ell^p$-improving problem for discrete curves}

\author{Spyridon Dendrinos}
\address[S. Dendrinos]{School of Mathematical Sciences, University College Cork, Western Gateway Building, Western Road, Cork, Ireland }
\email{sd@ucc.ie}

\author{Kevin Hughes}
\address[K. Hughes]{School of Mathematics, The University of Bristol, Fry Building, Woodland Road, Bristol, BS8 1UG, UK, and the Heilbronn Insitute for Mathematical Research, Bristol, UK}
\email{khughes.math@gmail.com}

\author{Marco Vitturi}
\address[M. Vitturi]{School of Mathematical Sciences, University College Cork, Western Gateway Building, Western Road, Cork, Ireland}
\email{marco.vitturi@ucc.ie}

\begin{document}

\begin{abstract}
We apply Christ's method of refinements to the $\ell^p$-improving problem for discrete averages $\mathcal{A}_N$ along polynomial curves in $\mathbb{Z}^d$. Combined with certain elementary estimates for the number of solutions to certain special systems of diophantine equations, we obtain some restricted weak-type $p \to p'$ estimates for the averages $\mathcal{A}_N$ in the subcritical regime. The dependence on $N$ of the constants here obtained is sharp, except maybe for an $\epsilon$-loss.
\end{abstract}

\maketitle

\section{Introduction}
We consider here polynomial curves in $\mathbb{Z}^d$ for $d\geq 1$ given by
\[ \gamma(n) = (P_1(n),\ldots, P_d(n)) \]
with $P_1,\ldots,P_d$ univariate polynomials in $\mathbb{Z}[X]$ with the property that their degrees are \emph{separated}, by which we mean that $\deg P_j < \deg P_{j+1}$ for all $j \in \{1,\ldots, d-1\}$. The prototypical example of such  curves is the \emph{moment curve}
\[ \Gamma(n) :=(n,n^2,n^3,\ldots, n^d). \]
To each curve $\gamma$ one can associate the sequence of discrete (forward) averages in $\mathbb{Z}^d$ along the curve given by
\[ \mathcal{A}^{\gamma}_{N}f (\bm{x}) := \frac{1}{N}\sum_{n=1}^{N} f(\bm{x}-\gamma(n)), \]
where $\bm{x}\in\mathbb{Z}^d$; we will typically omit the superscript if no confusion ensues. In analogy with the $L^p$-improving problem for continuous averages, one is interested in studying the $\ell^p \to \ell^q$ mapping properties of the operators $\mathcal{A}_N$ -- in particular, one would like to explore to what extent inequalities of the form
\begin{equation}
\|\mathcal{A}_N f\|_{\ell^q(\mathbb{Z}^d)} \leq C_{p,q,\gamma}(N) \|f\|_{\ell^p(\mathbb{Z})} \label{eq:generic_improving_estimate}
\end{equation}
can hold and what is the (asymptotically) smallest $C_{p,q,\gamma}(N)$ constant for which \eqref{eq:generic_improving_estimate} holds. In this regard, one is led to the following conjecture.
%
%
\begin{conjecture}\label{main_conjecture}
Let $D = D_{\gamma} := \sum_{j = 1}^{d} \mathrm{deg} P_j$ denote the \emph{total degree} of the curve $\gamma$. For any pair of exponents $p,q$ such that $q \geq p$ and for every $\epsilon > 0$
the estimate
\begin{equation}\label{eq:conjectured_estimate}
\|\mathcal{A}_N f\|_{\ell^{q}(\mathbb{Z}^d)} \lesssim_{\epsilon} N^{\epsilon}(N^{ - D(1/p - 1/q)} + N^{-1/{q'}} + N^{-1/p}) \|f\|_{\ell^{p}(\mathbb{Z}^d)}
\end{equation}
holds for every $f \in \ell^p(\mathbb{Z}^d)$. 
\end{conjecture}
That the condition $q \geq p$ is necessary can be seen by a standard argument due to H\"{o}rmander using the translation invariance of the operators $\mathcal{A}_N$ and the fact that they are local. To show that the best constant $C_{p,q,\gamma}(N)$ for which \eqref{eq:generic_improving_estimate} can hold is at least as large as the right-hand side of \eqref{eq:conjectured_estimate} it suffices to test against some standard examples: in particular, testing against a Dirac delta function one obtains $C_{p,q,\gamma}(N) \gtrsim N^{-1/{q'}}$, while testing against the characteristic function of the image\footnote{This example is simply the dual to the Dirac delta one.} $\gamma([1,N])$ one obtains $C_{p,q,\gamma}(N) \gtrsim N^{-1/p}$; finally, the remaining condition is obtained by testing against the characteristic function of (a suitable dilate of) a parabolic box of sidelengths $\sim N^{\mathrm{deg} P_1} \times \ldots \times N^{\mathrm{deg} P_d}$. Thus the conjecture states that the necessary powers of $N$ given by such examples are also sufficient.
\begin{remark}
The $\epsilon$-loss in the exponents of $N$ does not arise from the aforementioned examples and has been included out of an abundance of caution. It is possibly absent from the true estimates and can sometimes be removed by suitable $\epsilon$-removal lemmata, but we will not concern ourselves with such questions here.
\end{remark}
If an estimate of the form \eqref{eq:generic_improving_estimate} is shown to hold for a certain pair $p,q$ with constant $C_{p,q,\gamma}(N)$ as in \eqref{eq:conjectured_estimate}, we will say that the estimate is \emph{optimal}. In this paper we shall be concerned exclusively with optimal estimates. Moreover, the regime in which the term $N^{- D(1/p - 1/q)}$ at the right-hand side of \eqref{eq:conjectured_estimate} dominates over the other two will be called the \emph{supercritical regime}: by inspection, it consists of those exponents $p,q$ such that 
\begin{equation}\label{eq:description_of_supercritical_regime}
\begin{cases}
& q \geq p, \\
& \frac{D}{q} > \frac{D-1}{p}, \\
& \frac{D}{p'} > \frac{D-1}{q'};
\end{cases}
\end{equation}
on a $(1/p, 1/q)$ diagram, it corresponds to a triangle with one side given by the $1/p = 1/q$ line and the vertex opposite to it given by the critical endpoint $(1/p_c, 1/{p_c'}) = \big(\frac{D}{2D - 1}, \frac{D-1}{2D-1}\big)$ (that is, $p_c = 2 - \frac{1}{D}$). The complement of the supercritical regime in the $q \geq p$ range will be called the \emph{subcritical regime}. Estimates for exponents $p,q$ in either regime will be named accordingly. See Figure \ref{fig:regimes_riesz_diagram} for a pictorial representation of the regimes on a $(1/p, 1/q)$ diagram.
\begin{remark}
There seems to be a certain difference between the discrete and continuous case: in the former the supercritical regime is conjectured to correspond to a triangle in the $(1/p,1/q)$ diagram, while in the latter it is known (for suitably regular curves or for suitably weighted averages) that the range of boundedness corresponds to a trapezoid instead (when $d>2$).
\end{remark}
%
%
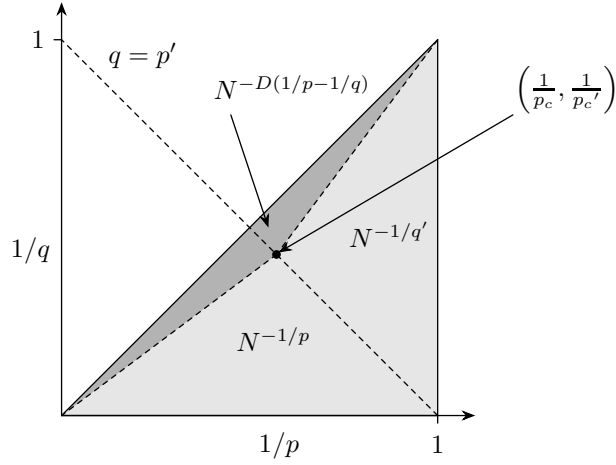
\begin{figure}[ht]
\centering
\begin{tikzpicture}[line cap=round,line join=round,>=Stealth,x=1cm,y=1cm,scale=5]
\clip(-0.45,-0.2) rectangle (1.5,1.2);
\fill[line width=0.5pt,color=black,fill=black,fill opacity=0.3] (0,0) -- (1,1) -- (4/7,3/7) -- cycle;
\fill[line width=0.5pt,color=black,fill=black,fill opacity=0.1] (0,0) -- (1,0) -- (4/7,3/7) -- cycle;
\fill[line width=0.5pt,color=black,fill=black,fill opacity=0.1] (1,1) -- (1,0) -- (4/7,3/7) -- cycle;
\draw [->,line width=0.5pt] (0,0) -- (1.1,0);
\draw [line width=0.5pt,color=black] (0,0)-- (1,1);
\draw [line width=0.5pt,color=black,dash pattern=on 2pt off 2pt] (1,1)-- (4/7,3/7);
\draw [line width=0.5pt,color=black,dash pattern=on 2pt off 2pt] (4/7,3/7)-- (0,0);
\draw [->,line width=0.5pt] (0,0) -- (0,1.1);
\draw [line width=0.5pt,color=black] (0,0)-- (1,0);
\draw [line width=0.5pt,color=black] (1,1)-- (1,0);
\draw [line width=0.5pt,dash pattern=on 2pt off 2pt] (1,0)-- (0,1);
\draw (0.5,-0.02) node[anchor=north west] {$1/p$};
\draw (-0.16,0.5) node[anchor=north west] {$1/q$};
\draw (1,-0.03) node[anchor=north] {$1$};
\draw (-0.02,1) node[anchor=east] {$1$};
\draw [line width=0.5pt] (1,-0.02)-- (1,0);
\draw [line width=0.5pt] (0,1)-- (-0.02,1);
\draw (0.44,0.26) node[anchor=north west] {$N^{-1/p}$};
\draw (0.74,0.54) node[anchor=north west] {$N^{-1/{q'}}$};
\draw (0.38,0.93) node[anchor=north west] {$N^{-D(1/p - 1/q)}$};
\draw (0.1,0.9) node[anchor=south west] {$q=p'$};
\draw [black,fill] (4/7,3/7) circle [radius=0.01];
\draw [->,line width=0.5pt] (0.45,0.8) -- (0.55,0.5);
\draw [->,line width=0.5pt] (1.2,0.8) -- (4/7+0.005,3/7+0.005);
\draw (1.18,0.78) node[anchor=south west] {$\Big(\frac{1}{p_c}, \frac{1}{{p_c}'}\Big)$};
\end{tikzpicture}
\caption{\footnotesize A pictorial illustration of Conjecture \ref{main_conjecture}. The supercritical regime corresponds to the darker triangle, which consists of the points $(1/p,1/q)$ for which the term $N^{-D(1/p - 1/q)}$ dominates over the other two in the right-hand side of \eqref{eq:conjectured_estimate}, as indicated. The tip of the triangle corresponds to the critical estimate. The lower lighter triangle below the $q = p'$ line corresponds to the part of the subcritical regime in which the term $N^{-1/p}$ dominates; the lighter triangle above the $q = p'$ line corresponds to the part of the subcritical regime in which the term $N^{-1/{q'}}$ dominates.} \label{fig:regimes_riesz_diagram}
\end{figure}
%
%
A first incarnation of the problem considered here has been the problem of studying the $\ell^p \to \ell^q$ mapping properties of discrete fractional integrals along discrete varieties (see \cite{HanKovacLaceyMadridYang} for why the two are quite related). This issue has attracted a great deal of attention over the years -- see \cite{IonescuWainger,Kim,Oberlin,Pierce1,Pierce2,SteinWainger1,SteinWainger2} for a selection of works in the area. The $\ell^p$-improving problem for discrete averages $\mathcal{A}_N$ proper is more recent but has seen a certain degree of activity lately. In particular, in \cite{HanLaceyYang} Han, Lacey and Yang have studied the $\ell^p$-improving properties of the averages $\mathcal{A}^{\gamma}_N$ along the polynomial $\gamma(n) = n^2$, while in \cite{HanKovacLaceyMadridYang} the same authors together with Kova\v{c} and Madrid have studied the case of $\gamma = \Gamma$ the moment curve. In the latter, using the celebrated solution to the Vinogradov Mean Value conjecture of Wooley \cite{Wooley} for $d=3$ and of Bourgain, Demeter and Guth \cite{BourgainDemeterGuth} for arbitrary $d$, they have shown for the moment curve the optimal supercritical estimates \eqref{eq:conjectured_estimate} for exponents $p,q$ such that 
\begin{equation*}
\begin{cases}
& q \geq p, \\
& \frac{D+1/2}{q} > \frac{D-1/2}{p}, \\
& \frac{D+1/2}{p'} > \frac{D-1/2}{q'},
\end{cases}
\end{equation*}
which is the (strict) subset of the supercritical regime given by interpolation of the trivial $\ell^p \to \ell^p$ inequalities with the optimal endpoint $\ell^{p_0} \to \ell^{{p_0}'}$ inequality with $p_0 = \frac{4D}{2D+1}$, where here $D = D_{\Gamma} = \frac{d(d+1)}{2}$ (see however Section \ref{section_HY_argument} for why the range they obtain is actually larger). By a transference argument they also obtained some supercritical optimal estimates for the curve given by the monomial $\gamma(n) = n^k$; however, along the $q = p'$ line the transference argument yields the optimal $\ell^p \to \ell^{p'}$ inequality only for $2 \geq p \geq 2 - O(1/k^2)$, while the conjectured endpoint is $2 - 1/k$.\par
We also mention the recent work of Dasu, Demeter and Langowski \cite{DasuDemeterLangowski} in which they have completely solved the analogous $\ell^p$-improving problem for the discrete paraboloid in $\mathbb{Z}^d$ for any $d \geq 2$. The obvious connection with the case considered here is given by $d=2$, in which the paraboloid is simply the parabola $\gamma(n) = (n,n^2)$. In here we will reprove their endpoint result for the parabola (see \textit{(\ref{main_case_2})} of Theorem \ref{main_theorem} below).\par
Paraboloids are not the only discrete hypersurfaces for which the $\ell^p$-improving problem for the associated averages has been studied -- see \cite{Anderson,Hughes,KeslerLacey} for work on the discrete sphere.
\subsection{Main results}
All the aforementioned papers \cite{DasuDemeterLangowski,HanKovacLaceyMadridYang,HanLaceyYang} deal with estimates for $\mathcal{A}_N$ in the supercritical regime. In this work we concentrate instead on the subcritical regime and prove some new optimal subcritical estimates for certain classes of curves. Our results can be summarised in the restricted weak-type estimates on the $q = p'$ line listed in the following theorem.
\begin{theorem}\label{main_theorem}
Let $\gamma$ be a polynomial curve with separated degrees. Then the following optimal inequalities hold:
\begin{enumerate}[(i)]
\item \label{main_case_1}Suppose that $d \geq 1$ and that $\gamma(n)$ is not a single linear polynomial. Then we have for any finite sets $E,F \subset \mathbb{Z}^d$
\begin{equation}
\langle \mathcal{A}_N \mathbf{1}_E, \mathbf{1}_F \rangle \lesssim_{\gamma,\epsilon} N^{-2/3 + \epsilon} |E|^{2/3} |F|^{2/3}.  \label{eq:3/2->3_estimate}
\end{equation}
\item \label{main_case_2}Suppose that $d\geq 2$ and that the first component of $\gamma$ is a linear polynomial. Then we have for any finite sets $E,F \subset \mathbb{Z}^d$
\begin{equation}
\langle \mathcal{A}_N \mathbf{1}_E, \mathbf{1}_F \rangle \lesssim_{\gamma,\epsilon} N^{-3/5 + \epsilon} |E|^{3/5} |F|^{3/5}. \label{eq:5/3->5/2_estimate}
\end{equation}
\item \label{main_case_3}Suppose that $d\geq 3$ and the first three components $P_1,P_2,P_3$ of $\gamma$ have degrees respectively equal to $1,2$ and $3$. Then we have for any finite sets $E,F \subset \mathbb{Z}^d$
\begin{equation}
\langle \mathcal{A}_N \mathbf{1}_E, \mathbf{1}_F \rangle \lesssim_{\gamma,\epsilon} N^{-4/7 + \epsilon} |E|^{4/7} |F|^{4/7}. \label{eq:7/4->7/3_estimate}
\end{equation}
\end{enumerate}
\end{theorem}
\begin{remark}
Inspection of the proof of Theorem \ref{main_theorem} reveals that the $\epsilon$-losses in the form of the factors $N^{\epsilon}$ above can be more precisely quantified to be of the form $\lesssim e^{C \frac{\log N}{\log \log N}}$ for some $C>0$.
\end{remark}
\begin{remark}
Observe that, if $Q \in \mathbb{Z}[X]$ is non-constant, \eqref{eq:3/2->3_estimate} holds equally true for curves $\gamma$ and $\gamma \circ Q$ (and is always optimal). From the statement of case \textit{(\ref{main_case_2})}, it appears at first sight that we no longer have this freedom of composition for \eqref{eq:5/3->5/2_estimate}. However, further inspection of the proof reveals that case \textit{(\ref{main_case_2})} of Theorem \ref{main_theorem} admits the following extension. Let $\mathfrak{X} \subset \mathbb{Z}$ be a sequence and let $\mathfrak{X}_N := \mathfrak{X} \cap [1,N]$ for any $N>0$; we can define the \emph{averages along $\gamma$ restricted to $\mathfrak{X}$} as 
\[ \mathcal{A}^{\gamma}_{\mathfrak{X},N}f(\bm{x}) := \frac{1}{|\mathfrak{X}_N|} \sum_{m\in \mathfrak{X}_N} f(\bm{x} - \gamma(m)). \]
When $Q \in \mathbb{Z}[X]$ is a non-constant univariate polynomial\footnote{We take the leading coefficient to be positive, for convenience.} and $\mathfrak{X}^Q$ the set of its values, that is $\mathfrak{X}^Q= \{ Q(n) : n \in \mathbb{Z} \}$, the operator $\mathcal{A}^{\gamma}_{\mathfrak{X}^Q,N}$ is essentially the same as $\mathcal{A}^{\gamma\circ Q}_{M}$ with $M \sim N^{1/\deg Q}$ (notice that $|\mathfrak{X}^Q_N| \sim N^{1/\deg Q}$). In this situation, inspecting the proof of Theorem \ref{main_theorem} (and in particular that of Lemma \ref{lemma_number_of_solutions_case_2}) we see that case \textit{(\ref{main_case_2})} of Theorem \ref{main_theorem} extends to 
\[ \langle \mathcal{A}_{\mathfrak{X}^Q,N} \mathbf{1}_E, \mathbf{1}_F \rangle \lesssim_{\gamma,\epsilon} |\mathfrak{X}^Q_N|^{-3/5 + \epsilon} |E|^{3/5} |F|^{3/5} \]
(so that \eqref{eq:5/3->5/2_estimate} holds for $\gamma \circ Q$ as well). Testing against the Dirac delta function (or its dual example) verifies that this estimate is optimal.
\end{remark}
We stress that in each of the three cases in Theorem \ref{main_theorem} the estimate is indeed in the subcritical\footnote{Although \eqref{eq:3/2->3_estimate},\eqref{eq:5/3->5/2_estimate} can be critical in special cases; see Corollary \ref{corollary_main_theorem}.} regime: in the first case the critical exponent $p_c$ is at least $2 - 1/2 = 3/2$, in the second at least $2 - 1/3 = 5/3$ and in the third at least $2 - 1/6 = 11/6 > 7/4$.\par
Lorentz interpolation of each of the above estimates with the trivial $\ell^p \to \ell^{\infty}$ and $\ell^1 \to \ell^q$ inequalities yields a range of strong-type optimal subcritical estimates. If one has optimal estimates on the critical lines (that separate the super- and subcritical regimes; see Figure \ref{fig:regimes_riesz_diagram}), it is possible to interpolate with those as well and obtain an even larger range of subcritical estimates (see Section \ref{section_preliminaries_reductions}). In particular, when the exponents of estimates \eqref{eq:3/2->3_estimate} and \eqref{eq:5/3->5/2_estimate} coincide with the critical exponents $p_c, {p_c}'$ we obtain by interpolation the full Conjecture \ref{main_conjecture} -- see the following corollary.
\begin{corollary}\label{corollary_main_theorem}
Conjecture \ref{main_conjecture} holds for all $q \geq p$ when:
\begin{enumerate}[(i)]
\item \label{corollary_case_quadratic}$d=1$ and $\gamma(n)$ is a quadratic polynomial;
\item \label{corollary_case_parabola}$d=2$ and $\gamma(n)$ is a parabola in $\mathbb{Z}^2$.
\end{enumerate}
\end{corollary}
Case \textit{(\ref{corollary_case_quadratic})} of the Corollary recovers the corresponding results of \cite{HanKovacLaceyMadridYang,HanLaceyYang}, while case \textit{(\ref{corollary_case_parabola})} recovers the 2D case of \cite{DasuDemeterLangowski}.
\begin{remark}
For the parabola, the $\epsilon$-removal technology of \cite{DasuDemeterLangowski} allows one to remove the $\epsilon$-loss (caused by interpolation) from the interior of the supercritical regime. The reach of such technology seems currently limited to the supercritical regime. We mention here in this regard that an easy application of our methods shows that when $d\geq 2$ and the first component of $\gamma$ is linear then \eqref{eq:3/2->3_estimate} holds without the $\epsilon$-loss; by interpolation, one obtains $\epsilon$-free optimal estimates in a strict subset of the subcritical range.
\end{remark}
The proof of Theorem \ref{main_theorem} relies on an adaptation to our arithmetic setting of the method of refinements introduced by Christ in \cite{Christ99} to study the continuous $L^p$-improving problem. The problem has been completely solved by the method in the case of continuous curves, see \cite{Christ99,DendrinosLaghiWright,Stovall}. Summarising briefly for the unaware reader, the essence of the method (at least in our adaptation) consists in a combinatorial reformulation of the restricted weak-type inequalities that translates them into lower bounds for $|E|$ in terms of parameters $\alpha,\beta$ such that $\alpha |F|= \beta|E| = \langle \mathcal{A}_N \bm{1}_E, \bm{1}_F\rangle$; this allows one to set up a ``flowing'' procedure based on $\gamma$, taking us alternatingly from the set $E$ to the set $F$ and vice versa. One can show that the procedure yields a ``large'' (in terms of $\alpha, \beta$) set of parameters $(n_1, m_1, \ldots, n_k,m_k) \in [1,N]^{2k}$ such that for a carefully chosen $\bm{y} \in E$ we have for all these parameters
\[ \bm{y} + \gamma(n_1) - \gamma(m_1) + \ldots + \gamma(n_k) - \gamma(m_k) \in E. \]
If one views this expression as a map, in our discrete context it is possible to obtain a lower bound for $|E|$ by estimating the multiplicity of the map; this lower bound then translates back into restricted weak-type estimates. The multiplicity estimates translate quite directly into a classical number-theory problem -- that of bounding the number of solutions to certain diophantine equations.\par
While preparing this manuscript we became aware that -- not surprisingly -- the method of refinements has been applied to such discrete questions before. In fact, Oberlin \cite{Oberlin} and Kim \cite{Kim} used it to prove $\ell^p \to \ell^q$ estimates for certain discrete fractional integrals along curves -- \cite{Kim} in particular provides somewhat general conditional statements. In contrast to \cite{Kim,Oberlin}, in our adaptation of the method we additionally prune the combinatorial tower of parameters so as to ensure that in our multiplicity bounds we never encounter equations of the form \eqref{eq:generic_diophantine_system} (see Section \ref{section_HY_argument}) but rather their inhomogeneous version. This pruning is crucial to us as the latter are expected to have fewer solutions than \eqref{eq:generic_diophantine_system}, since they do not admit solutions of diagonal type.\par
The rest of the paper is organised as follows: in Section \ref{section_preliminaries} we review certain basic facts about the averages $\mathcal{A}_N$ and describe how the aforementioned result of \cite{HanKovacLaceyMadridYang} is proven; in Section \ref{section_method_of_refinements_setup} we develop an arithmetic version of the method of refinements as needed in our case, with which the proof of Theorem \ref{main_theorem} is reduced to proving bounds for the number of solutions to certain diophantine systems that arise in the process; in Section \ref{section_number_of_solutions} we prove such bounds by elementary arguments, thus completing the proof of Theorem \ref{main_theorem}.
%
%
\subsection*{Notation and basic facts}
Throughout this manuscript we use $A \lesssim B$ to denote the inequality $A \leq C B$ for some suppressed constant $C>0$; $A \sim B$ means $A \lesssim B$ and $B \lesssim A$. When the suppressed constant depends on a certain list $\mathcal{L}$ of parameters we highlight this by writing $A \lesssim_{\mathcal{L}} B$. Moreover, in conditional statements we will use $A \gg B$ to denote the inequality $A \geq C B$ for some \emph{sufficiently large} constant $C>0$.\par
We use $[1,N]$ as shorthand for the set of integers $\{1, \ldots, N\}$. If $E \subset \mathbb{Z}^d$ then $|E|$ denotes its cardinality.\par
In Section \ref{section_number_of_solutions} we will repeatedly make use of the so-called divisor bound, which states that the number of distinct divisors of $n \neq 0$ is bounded by $\lesssim e^{C \frac{\log n}{\log \log n}}$; however, we will limit ourselves to the weaker version that states that for every $\epsilon > 0$ the number of divisors of $n\lesssim N$ is $\lesssim_\epsilon N^\epsilon$.
\section{Preliminaries}\label{section_preliminaries}
In this section we record some observations about the affine structure of the problem and then discuss how one can obtain estimates of the form \eqref{eq:generic_improving_estimate} by a simple Hausdorff-Young argument. The former will allow us to reduce case \textit{(\ref{main_case_1})} of Theorem \ref{main_theorem} to curves $\gamma$ of the form $P(n)$, case \textit{(\ref{main_case_2})} of Theorem \ref{main_theorem} to curves $\gamma$ of the form $(n,P(n))$ and case \textit{(\ref{main_case_3})} to $\gamma$ the moment curve in $d=3$, that is $\Gamma(n) = (n,n^2,n^3)$. The latter discussion will provide some context and allow us to illustrate what is the range in which Conjecture \ref{main_conjecture} is currently known to hold for the moment curve.
\subsection{Affine transformations and projections}\label{section_preliminaries_reductions}
In the discrete context an affine transformation that maps $\mathbb{Z}^d$ into itself does not necessarily have an inverse, that which hampers the usual change of variable arguments. Consider however the following special examples of linear transformations (we ignore the translations since they are harmless): 
\begin{enumerate}[(i)]
\item a linear transformation $T : \mathbb{Z}^d \to \mathbb{Z}^d$ such that 
\[ T(x_1, \ldots, x_d) = (a_1 x_1, \ldots, a_d x_d)  \quad \forall \bm{x} \in \mathbb{Z}^d \]
for some non-zero integers $a_1, \ldots, a_d$, which we will refer to as \emph{integer dilation}; 
\item a linear transformation $T : \mathbb{Z}^d \to \mathbb{Z}^d$ such that 
\[ T(x_1, \ldots, x_d) = (x_1, \ldots, x_{j-1}, x_j - b x_k, x_{j+1},\ldots, x_d) \quad \forall \bm{x} \in \mathbb{Z}^d \]
for some integer $b$ and $k \neq j$, which we will refer to as \emph{integer shear}.
\end{enumerate}
Then the following still holds.
\begin{lemma}\label{lemma_affine_transformations}
Let $T : \mathbb{Z}^d \to \mathbb{Z}^d$ be a linear transformation obtained by the composition of integer dilations and integer shears. Then for any curve $\gamma$ and any $N,p,q$ such that $q \geq p$ we have for the averages associated to $T\gamma$
\[ \| \mathcal{A}^{T\gamma}_N \|_{\ell^p(\mathbb{Z}^d) \to \ell^q(\mathbb{Z}^d)} \leq  \| \mathcal{A}^{\gamma}_N \|_{\ell^p(\mathbb{Z}^d) \to \ell^q(\mathbb{Z}^d)}, \]
and similarly for the restricted weak-type norms.
\end{lemma}
\begin{proof}
We consider here only the strong operator norm, since the result for restricted weak-type norms requires only trivial modifications.\par
It clearly suffices to check for a single integer dilation and a single integer shear separately. Abandoning for a moment our convention about the ordering of the polynomial degrees in the components of $\gamma$, we can assume that 
\[ T(x_1, \ldots, x_d) = (a x_1, x_2 \ldots, x_d) \]
or 
\[ T(x_1, \ldots, x_d) = (x_1 - b x_2, x_2, x_3, \ldots, x_d). \]
In the first case, write $\bm{x} = (x_1,\bm{x}') \in \mathbb{Z} \times \mathbb{Z}^{d-1}$ and let $z,r$ be the unique integers such that $x_1 = az + r$ with $0\leq r<a$; then if we let $g_r(s,\bm{y}) := f(as + r, \bm{y})$
we have 
\[ \mathcal{A}^{T\gamma}_{N} f(x_1, \bm{x}') = \mathcal{A}^{\gamma}_{N} g_r(z,\bm{x}'). \]
Therefore we have 
\begin{align*}
\sum_{\bm{x} \in \mathbb{Z}^d} |\mathcal{A}^{T\gamma}_{N} f(\bm{x})|^q & = \sum_{r=0}^{a-1} \sum_{(z,\bm{x}') \in \mathbb{Z} \times \mathbb{Z}^{d-1}} |\mathcal{A}^{\gamma}_{N} g_r(z,\bm{x}')|^q \\
& \leq \|\mathcal{A}^{\gamma}_{N}\|_{p \to q}^q \sum_{r = 0}^{a-1}\Big(\sum_{(z,\bm{y}) \in \mathbb{Z} \times \mathbb{Z}^{d-1}} |g_r(z,\bm{y})|^p \Big)^{q/p} \\
& \leq \|\mathcal{A}^{\gamma}_{N}\|_{p \to q}^q \Big( \sum_{r = 0}^{a-1}\sum_{(z,\bm{y}) \in \mathbb{Z} \times \mathbb{Z}^{d-1}} |g_r(z,\bm{y})|^p \Big)^{q/p} = \|\mathcal{A}^{\gamma}_{N}\|_{p \to q}^q \|f\|_{\ell^p(\mathbb{Z}^d)}^q,
\end{align*}
where we have used the fact that for $q \geq p$ the $\ell^q$ norm is smaller than the $\ell^p$ one. This proves the Lemma for integer dilations.\par
In the second case, the transformation $T$ is a linear bijection over $\mathbb{Z}^d$ with well-defined inverse. A standard change of variables argument (notice that integer shears leave the $\ell^p$ norms unchanged) then concludes the proof of the Lemma.
\end{proof}
It is an immediate consequence of Lemma \ref{lemma_affine_transformations} that it will suffice to prove case \textit{(\ref{main_case_2})} of Theorem \ref{main_theorem} for curves of the form $\gamma(n) = (n, P_2(n), \ldots, P_d(n))$; moreover, it will similarly suffice to prove case \textit{(\ref{main_case_3})} of Theorem \ref{main_theorem} for curves of the form $\gamma(n) = (n,n^2,n^3, P_4(n), \ldots, P_d(n))$. However, as anticipated, a further reduction is possible -- we encapsulate it in the following lemma.
\begin{lemma}\label{lemma_projections}
Let $\gamma_1, \gamma_2$ be polynomial curves mapping into $\mathbb{Z}^{d_1}, \mathbb{Z}^{d_2}$ respectively; $(\gamma_1(n),\gamma_2(n))$ is then a polynomial curve mapping into $\mathbb{Z}^{d_1} \times \mathbb{Z}^{d_2}$. If $q \geq p$ then we have 
\[ \| \mathcal{A}^{(\gamma_1,\gamma_2)}_{N}\|_{\ell^{p}(\mathbb{Z}^{d_1} \times \mathbb{Z}^{d_2}) \to \ell^q(\mathbb{Z}^{d_1} \times \mathbb{Z}^{d_2})} \leq \| \mathcal{A}^{\gamma_1}_{N} \|_{\ell^p(\mathbb{Z}^{d_1}) \to \ell^q(\mathbb{Z}^{d_1})}, \]
and similarly for the restricted weak-type norms and for $\gamma_2$ in place of $\gamma_1$.
\end{lemma}
The lemma is an easy consequence of Minkowski's inequality and the nesting of the $\ell^p$ norms, and thus we omit the proof.\par
An immediate consequence of Lemma \ref{lemma_affine_transformations} and Lemma \ref{lemma_projections} is therefore that it will suffice:
\begin{itemize}
\item to prove case \textit{(\ref{main_case_1})} of Theorem \ref{main_theorem} in the case $\gamma(n) = P(n)$ a univariate polynomial with $\deg P \geq 2$;
\item to prove case \textit{(\ref{main_case_2})} of Theorem \ref{main_theorem} in the case $\gamma(n) = (n,P(n))$ with $P$ a univariate polynomial with $\deg P \geq 2$.
\item to prove case \textit{(\ref{main_case_3})} of Theorem \ref{main_theorem} in the case $\gamma(n) = (n,n^2,n^3)$. 
\end{itemize}
\subsection{Estimates using Hausdorff-Young}\label{section_HY_argument}
The operators $\mathcal{A}_N$ are convolution operators; in particular, if we let 
\[ \mu_N := \frac{1}{N} \sum_{n=1}^{N} \delta_{\gamma(n)} \]
we have explicitly $\mathcal{A}_N f = f \ast \mu_N$. When $p$ is of the form $\frac{4s}{2s+1}$ (for $s \geq 1/2$), so that $1/p - 1/{p'} = 1/ 2s$, one can then argue as follows: 
\begin{align*}
\|\mathcal{A}_n f\|_{\ell^{p'}(\mathbb{Z}^d)} &= \|f \ast \mu_N \|_{\ell^{p'}(\mathbb{Z}^d)} \leq \| \widehat{f} \cdot \widehat{\mu_N} \|_{L^{p}(\mathbb{T}^d)} \\
& \leq \| \widehat{f}\|_{L^{p'}(\mathbb{T}^d)} \|\widehat{\mu_N}\|_{L^{2s}(\mathbb{T}^d)} \leq \|f\|_{\ell^{p}(\mathbb{Z}^d)} \|\widehat{\mu_N}\|_{L^{2s}(\mathbb{T}^d)},
\end{align*}
where we have used the Hausdorff-Young inequality twice\footnote{Once for the Fourier transform and once for its inverse.}. A bound for $\|\widehat{\mu_N}\|_{L^{2s}(\mathbb{T}^d)}$ will then result in an $\ell^p$-improving inequality of the form \eqref{eq:generic_improving_estimate}. By a standard orthogonality calculation one notices that when $s$ is integer $N^{2s} \|\widehat{\mu_N}\|_{L^{2s}(\mathbb{T}^d)}^{2s}$ coincides with the number of solutions with $n_j,m_j$ in $[1,N]$ to the system of $d$ diophantine equations 
\begin{equation}
\gamma(n_1) + \ldots + \gamma(n_s) = \gamma(m_1) + \ldots + \gamma(m_s);  \label{eq:generic_diophantine_system}
\end{equation}
a bound for such number will then result in an $\ell^p$-improving inequality as well.\par
When $\gamma = \Gamma$ the moment curve, the system of equations \eqref{eq:generic_diophantine_system} is the so-called Vinogradov diophantine system and it was shown in \cite{BourgainDemeterGuth,Wooley} that the number of solutions is $\lesssim_{\epsilon} N^{\epsilon}(N^s + N^{2s - D_{\Gamma}})$. Picking the critical value $s = D_{\Gamma}$ one obtains by the above argument, with $p_0 = 4D_{\Gamma} / (2D_{\Gamma} + 1)$ as before,
\begin{equation}
\|\mathcal{A}^{\Gamma}_{N}f\|_{\ell^{p_0'}(\mathbb{Z}^d)} \lesssim_{\epsilon} N^{-1/2 + \epsilon} \|f\|_{\ell^{p_0}(\mathbb{Z}^d)},  \label{eq:HKLMY_endpoint}
\end{equation}
which can be verified to be optimal. This is how the $\ell^p$-improving result in \cite{HanKovacLaceyMadridYang} was obtained.\par
We point out however that the argument can actually yield a little more than the above: indeed, by using Hausdorff-Young only once, we have for $q = 2s/(s-1)$
\begin{align*}
\|\mathcal{A}_N f\|_{\ell^{2s/(s-1)}(\mathbb{Z}^d)} & \leq \| \widehat{f} \cdot \widehat{\mu_N}\|_{L^{2s/(s+1)}(\mathbb{T}^d)} \\
& \leq \|\widehat{f}\|_{L^2(\mathbb{T}^d)} \|\widehat{\mu_N}\|_{L^{2s}(\mathbb{T}^d)} = \|f\|_{\ell^2(\mathbb{Z}^d)} \|\widehat{\mu_N}\|_{L^{2s}(\mathbb{T}^d)};
\end{align*}
when $\gamma=\Gamma$ and $s=D_{\Gamma}$ we obtain thus
\begin{equation}
\|\mathcal{A}^{\Gamma}_N f\|_{\ell^{2D_{\Gamma}/(D_{\Gamma}-1)}(\mathbb{Z}^d)} \lesssim_{\epsilon} N^{-1/2+\epsilon} \|f\|_{\ell^2(\mathbb{Z}^d)},  \label{eq:HY_critical_line_estimate}
\end{equation}
which implies estimate \eqref{eq:HKLMY_endpoint} by interpolation with its own dual. It can be verified that estimate \eqref{eq:HY_critical_line_estimate} is not only optimal but it is on the critical line $D/q = (D-1)/p$ (see \eqref{eq:description_of_supercritical_regime} and Figure \ref{fig:regimes_riesz_diagram}). This allows one to interpolate also with the trivial $\ell^p \to \ell^\infty$ and $\ell^1 \to \ell^q$ inequalities, thus proving Conjecture \ref{main_conjecture} not only in a subset of the supercritical regime but also in a subset of the subcritical one. If one further interpolates these inequalities with \eqref{eq:7/4->7/3_estimate} of Theorem \ref{main_theorem} (when $d\geq 3$) the range obtained is as in Figures \ref{fig:range_moment_curve}, \ref{fig:range_moment_curve_zoom}.
%
%
\begin{figure}[ht]
\centering
\begin{tikzpicture}[line cap=round,line join=round,>=Stealth,x=1cm,y=1cm, scale=7]
\clip(-0.3,-0.11) rectangle (1.3,1.1);
\filldraw[line width=0.5pt,color=black,fill=black,fill opacity=0.1] (0,0) -- (13/24,11/24) -- (1,1) -- cycle;
\filldraw[line width=0.5pt,color=black,fill=black,fill opacity=0.1] (0,0) -- (0.5,5/12) -- (7/12,0.5) -- (1,1) -- cycle;
\filldraw[line width=0.5pt,color=black,fill=black,fill opacity=0.1] (0,0) -- (0.5,5/12) -- (1,0) -- cycle;
\filldraw[line width=0.5pt,color=black,fill=black,fill opacity=0.1] (0.5,5/12) -- (4/7,3/7) -- (7/12,0.5) -- (1,0) -- cycle;
\filldraw[line width=0.5pt,color=black,fill=black,fill opacity=0.1] (1,1) -- (7/12,0.5) -- (1,0) -- cycle;
\draw [->,line width=0.5pt] (0,0) -- (1.1,0);
\draw [->,line width=0.5pt] (0,0) -- (0,1.1);
\draw [line width=0.5pt,dash pattern=on 1pt off 1pt] (1,0)-- (0,1);
\draw [line width=0.5pt,dash pattern=on 1pt off 1pt] (0,0)-- (6/11,5/11);
\draw [line width=0.5pt,dash pattern=on 1pt off 1pt] (1,1)-- (6/11,5/11);
\draw [line width=0.5pt,color=black] (0,0)-- (13/24,11/24);
\draw [line width=0.5pt,color=black] (13/24,11/24)-- (1,1);
\draw [line width=0.5pt,color=black] (1,1)-- (0,0);
\draw [line width=0.5pt,color=black] (0,0)-- (0.5,5/12);
\draw [line width=0.5pt,color=black] (0.5,5/12)-- (7/12,0.5);
\draw [line width=0.5pt,color=black] (7/12,0.5)-- (1,1);
\draw [line width=0.5pt,color=black] (1,1)-- (0,0);
\draw [line width=0.5pt,color=black] (0,0)-- (0.5,5/12);
\draw [line width=0.5pt,color=black] (0.5,5/12)-- (1,0);
\draw [line width=0.5pt,color=black] (1,0)-- (0,0);
\draw [line width=0.5pt,color=black] (0.5,5/12)-- (4/7,3/7);
\draw [line width=0.5pt,color=black] (4/7,3/7)-- (7/12,0.5);
\draw [line width=0.5pt,color=black] (7/12,0.5)-- (1,0);
\draw [line width=0.5pt,color=black] (1,0)-- (0.5,5/12);
\draw [line width=0.5pt,color=black] (1,1)-- (7/12,0.5);
\draw [line width=0.5pt,color=black] (7/12,0.5)-- (1,0);
\draw [line width=0.5pt,color=black] (1,0)-- (1,1);
\draw [line width=0.5pt] (1,0)-- (1,-0.02);
\draw [line width=0.5pt] (0,1)-- (-0.02,1);
\draw (1,-0.04) node[anchor=north] {$1$};
\draw (0.5,-0.02) node[anchor=north] {$1/p$};
\draw (-0.02,0.5) node[anchor=east] {$1/q$};
\draw (0.15,0.85) node[anchor=south west] {$q=p'$};
\draw (-0.03,1) node[anchor=east] {$1$};
\end{tikzpicture}
\caption{\footnotesize The shaded area corresponds to the range of optimal estimates for the moment curve $\Gamma$ ($d\geq 3$) obtained by interpolating estimate \eqref{eq:HY_critical_line_estimate}, its dual and estimate \eqref{eq:7/4->7/3_estimate} with the trivial ones.} \label{fig:range_moment_curve}
\end{figure}
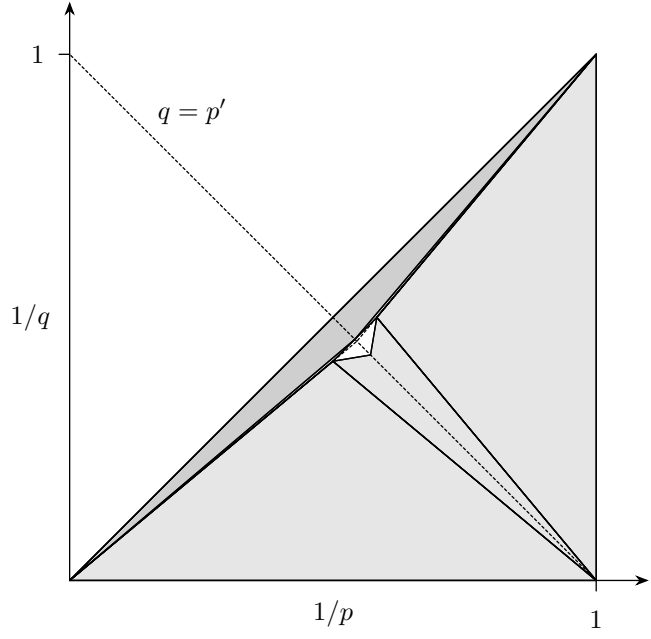
%
%
%
\begin{figure}[ht]
\centering
\begin{tikzpicture}[line cap=round,line join=round,>=Stealth,x=1cm,y=1cm, scale=55]
\clip(0.47,0.41) rectangle (0.62,0.51);
\filldraw[line width=0.5pt,color=black,fill=black,fill opacity=0.1] (0,0) -- (13/24,11/24) -- (1,1) -- cycle;
\filldraw[line width=0.5pt,color=black,fill=black,fill opacity=0.1] (0,0) -- (0.5,5/12) -- (7/12,0.5) -- (1,1) -- cycle;
\filldraw[line width=0.5pt,color=black,fill=black,fill opacity=0.1] (0,0) -- (0.5,5/12) -- (1,0) -- cycle;
\filldraw[line width=0.5pt,color=black,fill=black,fill opacity=0.1] (0.5,5/12) -- (4/7,3/7) -- (7/12,0.5) -- (1,0) -- cycle;
\filldraw[line width=0.5pt,color=black,fill=black,fill opacity=0.1] (1,1) -- (7/12,0.5) -- (1,0) -- cycle;
\draw [->,line width=0.5pt] (0,0) -- (1.1,0);
\draw [->,line width=0.5pt] (0,0) -- (0,1.1);
\draw [line width=0.5pt,dash pattern=on 1pt off 1pt] (1,0)-- (0,1);
\draw [line width=0.5pt,dash pattern=on 1pt off 1pt] (0,0)-- (6/11,5/11);
\draw [line width=0.5pt,dash pattern=on 1pt off 1pt] (1,1)-- (6/11,5/11);
\draw [line width=0.5pt,color=black] (0,0)-- (13/24,11/24);
\draw [line width=0.5pt,color=black] (13/24,11/24)-- (1,1);
\draw [line width=0.5pt,color=black] (1,1)-- (0,0);
\draw [line width=0.5pt,color=black] (0,0)-- (0.5,5/12);
\draw [line width=0.5pt,color=black] (0.5,5/12)-- (7/12,0.5);
\draw [line width=0.5pt,color=black] (7/12,0.5)-- (1,1);
\draw [line width=0.5pt,color=black] (1,1)-- (0,0);
\draw [line width=0.5pt,color=black] (0,0)-- (0.5,5/12);
\draw [line width=0.5pt,color=black] (0.5,5/12)-- (1,0);
\draw [line width=0.5pt,color=black] (1,0)-- (0,0);
\draw [line width=0.5pt,color=black] (0.5,5/12)-- (4/7,3/7);
\draw [line width=0.5pt,color=black] (4/7,3/7)-- (7/12,0.5);
\draw [line width=0.5pt,color=black] (7/12,0.5)-- (1,0);
\draw [line width=0.5pt,color=black] (1,0)-- (0.5,5/12);
\draw [line width=0.5pt,color=black] (1,1)-- (7/12,0.5);
\draw [line width=0.5pt,color=black] (7/12,0.5)-- (1,0);
\draw [line width=0.5pt,color=black] (1,0)-- (1,1);
\draw [line width=0.5pt] (1,0)-- (1,-0.01);
\draw [line width=0.5pt] (0,1)-- (-0.01,1);
\draw (1,-0.02) node[anchor=north west] {\tiny $1$};
\draw (0.5,-0.01) node[anchor=north west] {\tiny $1/p$};
\draw (-0.1,0.5) node[anchor=north west] {\tiny $1/q$};
\draw (0.1,0.9) node[anchor=north west] {\tiny $q=p'$};
\draw (0.593,0.47) node[anchor=north] {\tiny $\Big(\frac{D+1}{2D}, \frac{1}{2}\Big)$};
\draw (0.494,0.44) node[anchor=south] {\tiny $\Big(\frac{1}{2},\frac{D-1}{2D}\Big)$};
\draw (0.54,0.49) node[anchor=south] {\tiny $\Big(\frac{1}{p_0}, \frac{1}{{p_0}'}\Big) = \Big(\frac{2D+1}{4D},\frac{2D-1}{4D}\Big)$};
\draw (0.595,0.435) node[anchor=west] {\tiny $\Big(\frac{4}{7},\frac{3}{7}\Big)$};
\draw [->,line width=0.5pt] (0.494,0.44) -- (0.5,5/12);
\draw [->,line width=0.5pt] (0.54,0.49) -- (13/24,11/24);
\draw [->,line width=0.5pt] (0.593,0.47) -- (7/12,0.5);
\draw [->,line width=0.5pt] (0.595,0.435) -- (4/7,3/7);

\draw (0.554,0.421) node[anchor=north] {\tiny $\Big(\frac{1}{p_c}, \frac{1}{{p_c}'}\Big) =\Big(\frac{D}{2D-1},\frac{D-1}{2D-1}\Big)$};
\draw [->,line width=0.5pt] (0.555,0.421) -- (6/11,5/11);

\draw (-0.05,1) node[anchor=north west] {\tiny $1$};
\draw [fill=black] (13/24,11/24) circle (0.03pt);
\draw [fill=black] (6/11,5/11) circle (0.03pt);
\draw [fill=black] (7/12,0.5) circle (0.03pt);
\draw [fill=black] (4/7,3/7) circle (0.03pt);
\draw [fill=black] (0.5,5/12) circle (0.03pt);
\end{tikzpicture}
\caption{\footnotesize Zooming in on the region of Figure \ref{fig:range_moment_curve} near the critical point $\Big(\frac{D}{2D-1}, \frac{D-1}{2D-1}\Big)$.} \label{fig:range_moment_curve_zoom}
\end{figure}
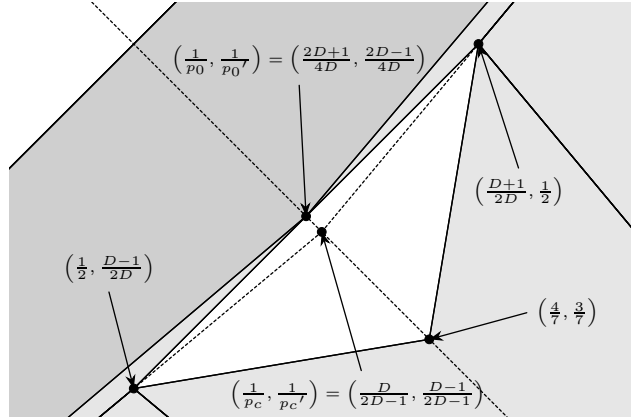
%
%
\par
It is conjectured from standard number-theoretical arguments that for a generic diophantine system of the form \eqref{eq:generic_diophantine_system} the number of solutions in $[1,N]^{2s}$ is controlled by the generalisation of the Vinogradov bound above, that is
\begin{equation}
\begin{aligned}
J_{s,\gamma}(N) := |\{& (n_1,\ldots,n_s, m_1,\ldots,m_s) \in [1,N]^{2s} : \\
& \gamma(n_1) + \ldots +\gamma(n_s) = \gamma(m_1) + \ldots + \gamma(m_s)\}| \\
& \hspace{10em}\lesssim_{\epsilon} N^{\epsilon} (N^s + N^{2s-D_{\gamma}}).
\end{aligned} \label{eq:conjectured_number_of_solutions}
\end{equation}
It is easily verified that if \eqref{eq:conjectured_number_of_solutions} holds for a certain $\gamma, s$ then the $\ell^2 \to \ell^{2s/(s-1)}$ bound obtained by the argument above is optimal (although it will only be critical if $s = D_{\gamma}$).\par
We record in Table \ref{table_estimates} below a list of optimal estimates for a few examples, obtained from known sharp number-theory estimates for $\|\widehat{\mu_N}\|_{L^{2s}(\mathbb{T}^d)}$ (some in the form \eqref{eq:conjectured_number_of_solutions}, as indicated).
\begin{table}[ht]
\centering
\begin{tabular}{|c|c|c|c|}
\hline
$\gamma(n)$ & $s$ & number-theory estimate & $\ell^2 \to \ell^q$ estimate \\
\hline
$n^3$ & 4 & $J_{s,\gamma}(N)\lesssim_\epsilon N^{5+\epsilon}$ (\cite{Vaughan86}) & $\|\mathcal{A}_N \|_{\ell^2(\mathbb{Z}) \to \ell^{8/3}(\mathbb{Z})} \lesssim_\epsilon N^{-3/8 + \epsilon}$ \\
\hline
$n^3$ & 2 & $J_{s,\gamma}(N)\lesssim_\epsilon N^{2+\epsilon}$  & $\|\mathcal{A}_N \|_{\ell^2(\mathbb{Z}) \to \ell^{4}(\mathbb{Z})} \lesssim_\epsilon N^{-1/2 + \epsilon}$ \\
\hline
$n^4$ & $\frac{23}{3}$ & $\|\widehat{\mu_N}\|_{L^{46/3}(\mathbb{T})} \lesssim_\epsilon N^{-6/23+\epsilon}$ (\cite{Wooley_nested}) & $\|\mathcal{A}_N \|_{\ell^2(\mathbb{Z}) \to \ell^{23/10}(\mathbb{Z})} \lesssim_\epsilon N^{-6/23 + \epsilon}$ \\
\hline
$n^5$ & $\frac{23}{2}$ & $\|\widehat{\mu_N}\|_{L^{23}(\mathbb{T})} \lesssim_\epsilon N^{-5/23+\epsilon}$ (\cite{Wooley_nested}) & $\|\mathcal{A}_N \|_{\ell^2(\mathbb{Z}) \to \ell^{46/21}(\mathbb{Z})} \lesssim_\epsilon N^{-5/23 + \epsilon}$ \\
\hline
$(n,n^3)$ & $\frac{9}{2}$ & $\|\widehat{\mu_N}\|_{L^9(\mathbb{T}^2)}\lesssim_\epsilon N^{-4/9+\epsilon}$ (\cite{Wooley_cubic}) & $\|\mathcal{A}_N \|_{\ell^2(\mathbb{Z}^2) \to \ell^{18/7}(\mathbb{Z}^2)} \lesssim_\epsilon N^{-4/9 + \epsilon}$ \\
\hline
\end{tabular}
\caption{Some optimal $\ell^2 \to \ell^q$ estimates obtained by the H-Y argument.} \label{table_estimates}
\vspace{-7mm}
\end{table}
\begin{remark}
It is clear that such simple Hausdorff-Young arguments can never give the critical endpoint of Conjecture \ref{main_conjecture}, even when using the strongest number-theory estimates available.
\end{remark}
\section{Arithmetic method of refinements for discrete curves}\label{section_method_of_refinements_setup}
As anticipated, our plan is to adapt the method of refinements of Christ (\cite{Christ99}) to our arithmetic setting and apply it to inequalities of the form \eqref{eq:generic_improving_estimate} (or rather, their restricted weak-type versions). In this section we will develop the setup and use it to reduce the proof of Theorem \ref{main_theorem} to certain number-theory estimates for systems of diophantine equations. We will find it more convenient to work with the un-normalised averages from now on, so we define
\[ \mathscr{A}^{\gamma}_N f(\bm{x}) := N \mathcal{A}^{\gamma}_{N} f(\bm{x}) = \sum_{n = 1}^{N} f(\bm{x} - \gamma(n)). \]
${\mathscr{A}_N}^{\ast}$ will denote the adjoint of $\mathscr{A}_N$ instead.
\subsection{Combinatorial reformulation of the estimates}\label{section_reformulation}
In order to apply the method of refinements to our problem we need to first reformulate the desired restricted weak-type estimates in an equivalent combinatorial fashion as follows.\par
If we let $E,F$ denote finite subsets of $\mathbb{Z}^d$, the restricted weak-type version of the conjectured inequality \eqref{eq:conjectured_estimate} is
\[ \langle \mathscr{A}_N \mathbf{1}_E, \mathbf{1}_F \rangle \lesssim_{\epsilon} N^{\epsilon} (N^{1-D(1/p - 1/q)} + N^{1/q} + N^{1/{p'}}) |E|^{1/p}|F|^{1/{q'}}. \]
If we introduce the quantities 
\[ \alpha := \frac{\langle \mathscr{A}_N \mathbf{1}_E, \mathbf{1}_F \rangle}{|F|}, \qquad \beta := \frac{\langle {\mathscr{A}_N}^{\ast}\mathbf{1}_F, \bm{1}_{E} \rangle}{|E|}, \]
and let $\frac{1}{r} := \frac{1}{p} - \frac{1}{q}$, we see after some calculations that the restricted weak-type inequality above in the supercritical regime is equivalently rewritten as 
\[ \alpha^{r/{q'}} \beta^{r/q} \lesssim_{\epsilon} N^{r-D+\epsilon} |E|. \]
If the exponents $p,q$ are in the subcritical regime instead, then the (conjectured) restricted weak-type inequality is equivalently rewritten as 
\[ \alpha^{r/{q'}} \beta^{r/q} \lesssim_{\epsilon} N^{\epsilon} (N^{r/q} + N^{r/{p'}}) |E|. \]
If we take $p = 2 - \frac{1}{k+1}$ for some $k < D$ and $q = p'$ the latter becomes in particular
\begin{equation}
\alpha^{k+1} \beta^{k} \lesssim_{\epsilon} N^{k+\epsilon} |E| \label{eq:refinement_inequality_subcritical_p_q}
\end{equation}
(which is in the subcritical regime if $k< D-1$ and is the critical endpoint estimate if $k = D-1$). The estimates in Theorem \ref{main_theorem} are precisely of this form.
\begin{remark}\label{remark_trivial_bound}
Observe that since we are considering only characteristic functions we have always $\alpha,\beta \leq N$.
\end{remark}
This reformulation of the desired inequalities has turned our task into that of proving suitable lower bounds for $|E|$ in terms of $\alpha,\beta, N$. The machinery developed in this Section will serve to achieve precisely this.
\subsection{Refinements of the sets  $E,F$}
The lemma presented in this subsection is well-known and a number of presentations exist in the literature -- it first appeared in \cite{Christ99}.\par
Let $E,F$ be finite subsets of $\mathbb{Z}^d$ and $\alpha,\beta$ be as defined in Section \ref{section_reformulation}. Observe that if we let
\begin{equation*}
F_1 := \big\{ \bm{x} \in F : \mathscr{A}_N \mathbf{1}_E(\bm{x}) > \frac{\alpha}{2} \big\}
\end{equation*}
then we have
\begin{equation}\label{eq:F_1_condition_ii} \langle \mathscr{A}_N \bm{1}_E, \bm{1}_{F_1} \rangle \geq \frac{1}{2}\langle \mathscr{A}_N \bm{1}_{E}, \bm{1}_{F} \rangle. 
\end{equation}
Indeed, one sees that  
\[ \langle \mathscr{A}_N \bm{1}_{E}, \bm{1}_{F\backslash F_1}\rangle \leq \frac{\alpha}{2}|F| = \frac{1}{2}\langle \mathscr{A}_N \bm{1}_{E}, \bm{1}_{F} \rangle, \] 
from which \eqref{eq:F_1_condition_ii} follows at once. Notice that in particular we have that $F_1 \neq \emptyset$. The observation extends easily to show that we can define iteratively (with $E_0 = E, F_0 = F$)
\begin{align*}
F_{j} &:= \big\{\bm{x} \in F_{j-1} : \mathscr{A}_N \mathbf{1}_{E_{j-1}}(\bm{x}) \gtrsim_j \alpha \big\}, \\
E_{j} &:= \big\{\bm{y} \in E_{j-1} : {\mathscr{A}_{N}}^{\ast}\mathbf{1}_{F_{j}}(\bm{y}) \gtrsim_j \beta \big\},
\end{align*}
for implicit constants decreasing sufficiently fast and obtain a sequence of sets as per the following lemma.
\begin{lemma}[Refinement lemma]\label{lemma_refinements}
Given $E, F$ finite subsets of $\mathbb{Z}^d$, there exists a sequence of subsets $E_j \subseteq E_0 := E$ and $F_j \subseteq F_0:= F$ such that we have for every $j$
\begin{enumerate}[(i)]
\item $F_j \subseteq F_{j-1}$;
\item \label{alpha_pointwise_lowerbound}for every $\bm{x} \in F_j$ we have $\mathscr{A}_N \bm{1}_{E_{j-1}}(\bm{x}) \gtrsim_j \alpha$;
\item \label{mass_preserved_1}$\langle \mathscr{A}_N \bm{1}_{E_{j-1}}, \bm{1}_{F_j}\rangle \gtrsim_j \langle \mathscr{A}_N \bm{1}_{E}, \bm{1}_{F} \rangle$;
\item $E_j \subseteq E_{j-1}$;
\item \label{beta_pointwise_lowerbound}for every $\bm{y} \in E_j$ we have ${\mathscr{A}_N}^{\ast}\bm{1}_{F_j}(\bm{y}) \gtrsim_j \beta$;
\item \label{mass_preserved_2}$\langle {\mathscr{A}_N}^{\ast} \bm{1}_{F_j}, \bm{1}_{E_j}\rangle \gtrsim_j \langle \mathscr{A}_N \bm{1}_{E}, \bm{1}_{F} \rangle$.
\end{enumerate}
\end{lemma}
We omit the easy proof.
\begin{remark}
Ultimately, properties \textit{(\ref{mass_preserved_1})} and \textit{(\ref{mass_preserved_2})} are only needed to show that the sets $E_j, F_j$ are not empty.
\end{remark}
\subsection{Flowing back and forth}\label{subsection_flowing}
Given the refined sets produced by Lemma \ref{lemma_refinements}, we consider now the parameters in which we are averaging.\par
If $\bm{y} \in E_{k}$ we have by \textit{(\ref{beta_pointwise_lowerbound})} of Lemma \ref{lemma_refinements} that ${\mathscr{A}_N}^{\ast}\bm{1}_{F_k}(\bm{y}) \gtrsim_j \beta$; but notice that 
\[ {\mathscr{A}_N}^{\ast}\bm{1}_{F_k}(\bm{y}) = \sum_{n=1}^{N} \bm{1}_{F_k}(\bm{y} + \gamma(n)) = |B^{\bm{y}}| \]
with
\[ B^{\bm{y}} := \{ n_1 \in [1,N] : \bm{y} + \gamma(n_1) \in F_k \}; \]
thus \textit{(\ref{beta_pointwise_lowerbound})} is a statement about the cardinality of the set of parameters $B^{\bm{y}} \subseteq [1,N]$ -- namely the lower bound $|B^{\bm{y}}| \gtrsim \beta$. Now if $\bm{y}\in E_{k}$ and $n_1\in B^{\bm{y}}$ as above we have $\bm{y}+\gamma(n_1) \in F_k$ and therefore we have by \textit{(\ref{alpha_pointwise_lowerbound})} of Lemma \ref{lemma_refinements} that $\mathscr{A}_N \bm{1}_{E_{k-1}}(\bm{y}+\gamma(n_1)) \gtrsim \alpha$; but again 
\[ \mathscr{A}_N \bm{1}_{E_{k-1}}(\bm{y}+\gamma(n_1)) = \sum_{m=1}^{N} \bm{1}_{E_{k-1}}(\bm{y}+\gamma(n_1)-\gamma(m)) =|A^{\bm{y}}_{n_1}| \]
with
\[ A^{\bm{y}}_{n_1} := \{ m_1 \in [1,N] : \bm{y}+\gamma(n_1)-\gamma(m_1) \in E_{k-1} \}, \]
so that \textit{(\ref{alpha_pointwise_lowerbound})} is also a statement about the cardinality of the sets of parameters $A^{\bm{y}}_{n_1} \subseteq [1,N]$ -- namely that $|A^{\bm{y}}_{n_1}|\gtrsim \alpha$ if $n_1 \in B^{\bm{y}}$.\par
We can clearly continue in this fashion and obtain a collection of \emph{slices}
\[ B^{\bm{y}}, A^{\bm{y}}_{n_1}, B^{\bm{y}}_{n_1,m_1}, A^{\bm{y}}_{n_1,m_1, n_2}, \ldots,B^{\bm{y}}_{n_1,m_1,\ldots,n_{k-1},m_{k-1}}, A^{\bm{y}}_{n_1,m_1,\ldots,m_{k-1},n_k}, \]
(each a subset of $[1,N]$ and each parametrised by the previous ones, which results in somewhat cumbersome notation) where we have defined iteratively
\begin{align*}
B^{\bm{y}}_{n_1,m_1,\ldots,n_{j},m_{j}} := \{n_{j+1} \in [1,N] : \bm{y} & +\gamma(n_1)-\gamma(m_1)+ \ldots \\
& + \gamma(n_{j}) - \gamma(m_j) + \gamma(n_{j+1}) \in F_{k-j} \}, \\
A^{\bm{y}}_{n_1,m_1,\ldots,m_j,n_{j+1}} := \{m_{j+1} \in [1,N] : \bm{y} & +\gamma(n_1)-\gamma(m_1)+ \ldots \\
& - \gamma(m_j) + \gamma(n_{j+1}) - \gamma(m_{j+1}) \in E_{k-j-1} \}.
\end{align*}
By Lemma \ref{lemma_refinements}, the slices have the fundamental properties that if one takes the chain of parameters $n_1, m_1, \ldots, n_k, m_k$ such that
\begin{equation}
n_1 \in B^{\bm{y}},\, m_1 \in A^{\bm{y}}_{n_1},\, n_2 \in B^{\bm{y}}_{n_1,m_1},\, \ldots,\, m_k \in A^{\bm{y}}_{n_1,m_1, \ldots, n_k},  \label{eq:chain_of_parameters}
\end{equation}
then we have lower bounds
\[ |B^{\bm{y}}_{n_1,\ldots,m_{j}}| \gtrsim \beta,  \qquad |A^{\bm{y}}_{n_1,\ldots,n_{j+1}}| \gtrsim \alpha, \]
for all $j$ and moreover we have 
\begin{align*} 
\bm{y} \in &\, E_{k}, \\
\bm{y} +\gamma(n_1) \in &\, F_{k}, \\
\bm{y}+\gamma(n_1)-\gamma(m_1) \in &\, E_{k-1}, \\
\vdots \hspace{2em} & \\
\bm{y}+\gamma(n_1)-\gamma(m_1) + \ldots +\gamma(n_k) \in &\, F_1, \\
\bm{y}+\gamma(n_1)-\gamma(m_1) + \ldots +\gamma(n_k)-\gamma(m_{k}) \in &\, E.
\end{align*}
In particular, we see that we are ``flowing'' between $E$ and $F$ with each step.
\subsection{Tower of parameters}\label{subsection_tower_construction}
The parameter slices defined in Section \ref{subsection_flowing} assemble naturally into the structure described below that is at the heart of the method of refinements. We will however prune one of the slices before assembling them, in order to enforce a certain crucial condition.\par
Let $k$ be as in the previous subsection and assume that $\alpha \gg 1$, depending on certain parameters introduced below\footnote{We will see in the proof of Theorem \ref{main_theorem} that this condition can always be enforced.}. We leave all slices undisturbed, safe for the last one, which we redefine to be 
\begin{align*}
\widetilde{A}^{\bm{y}}_{n_1,m_1,\ldots,m_{k-1},n_{k}} &:= A^{\bm{y}}_{n_1,m_1,\ldots,m_{k-1},n_{k}} \backslash \{ m_k \in [1,N] : \\
&\hspace{8em} \gamma(n_1)-\gamma(m_1) + \ldots +\gamma(n_k)-\gamma(m_{k}) = \bm{0} \}.
\end{align*}
Notice that in the set we are removing the variables $n_1, m_1, \ldots, n_k$ are fixed and thus the set consists of the common zeroes of certain univariate polynomials. The set has thus cardinality $\lesssim_{\gamma} 1$ and since $\alpha \gg 1$ we still have 
\[ |\widetilde{A}^{\bm{y}}_{n_1,m_1,\ldots,m_{k-1}, n_k}| \gtrsim \alpha. \]
\par
We will now define iteratively the set $\mathcal{T} \subseteq [1,N]^{2k}$ of sequences of parameters $(n_1,m_1,$ $\ldots,n_k,m_k)$ obtained by flowing back and forth as per Section \ref{subsection_flowing}. The set $\mathcal{T}$ is called the \emph{tower of parameters} and is defined as follows: let $S_1 := B^{\bm{y}}$ and $T_1:= \bigcup_{n_1\in S_1} \{n_1\}\times A^{\bm{y}}_{n_1}$, and let iteratively
\[ S_j := \bigcup_{\bm{t}\in T_{j-1}} \{\bm{t}\}\times B^{\bm{y}}_{\bm{t}}, \qquad
T_j := \bigcup_{\bm{s}\in S_{j}} \{\bm{s}\}\times A^{\bm{y}}_{\bm{s}}, \]
except for $T_k$ where in place of $A^{\bm{y}}_{\bm{s}}$ we use the pruned slice $\widetilde{A}^{\bm{y}}_{\bm{s}}$ instead. Then the tower of parameters is simply $\mathcal{T} := T_{k}$. The elements $(n_1,m_1, \ldots, n_k,m_k)$ of $\mathcal{T}$ are chains that satisfy \eqref{eq:chain_of_parameters}. We will say that $\mathcal{T}$ has been constructed by \emph{flowing $k$ times in each direction}.
\begin{remark}\label{remark_pruned_tower_property}
The pruning has had the effect of enforcing the condition that wherever we follow the flow starting at $\bm{y} \in E_k$ and given by $(n_1,m_1, \ldots, n_k,m_k)$ we never end up back at point $\bm{y}$. In other words, if $(n_1,m_1, \ldots, n_k,m_k) \in \mathcal{T}$ we have 
\[ \bm{y}+\gamma(n_1)-\gamma(m_1) + \ldots +\gamma(n_k)-\gamma(m_{k}) \neq \bm{y}. \]
\end{remark}
Observe that we can provide a lower bound for the cardinality of $\mathcal{T}$ in terms of $\alpha,\beta$: indeed, we have
\[ |\mathcal{T}| = |T_{k}| = \sum_{\bm{s}\in S_{k}} |\widetilde{A}^{\bm{y}}_{\bm{s}}| \gtrsim \alpha |S_k| \quad \text{ and } \quad |S_k| = \sum_{\bm{t}\in T_{k-1}} |B^{\bm{y}}_{\bm{t}}| \gtrsim \beta |T_{k-1}|; \]
iterating all the way to $S_1 = B^{\bm{y}}$, we obtain 
\begin{equation}\label{eq:lowerbound_tower}
|\mathcal{T}| \gtrsim \alpha^{k}\beta^{k}.
\end{equation}

\subsection{Lower bounds for $|E|$ and proof of Theorem \ref{main_theorem}}\label{section_lowerbound_and_proof}
With $\bm{y}$ as above, we now let $\Psi$ denote the map
\begin{equation}
\Psi(n_1,m_1,\ldots,n_k,m_k) := \bm{y}+\gamma(n_1)-\gamma(m_1) + \ldots +\gamma(n_k)-\gamma(m_k). \label{eq:map_Psi}
\end{equation}
The definition of the slices and of $\mathcal{T}$ show that $\Psi(\mathcal{T}) \subseteq E$ by construction, and therefore we have quite simply the lower bound 
\[ |\Psi(\mathcal{T})| \leq |E|. \]
This lower bound is of limited use in this form as in general it is not easy to compute $|\Psi(\mathcal{T})|$. One can however estimate it using the multiplicity of the map $\Psi$ over $\mathcal{T}$ in the following way. Given a mapping $\Phi: [1,N]^{s} \to \mathbb{Z}^d$ and a set $S \subseteq [1,N]^{s}$ we define the \emph{multiplicity} of $\Phi$ over $S$ to be 
\[ m(\Phi;S) := \max_{z \in \Phi(S)} |\Phi^{-1}(\{z\}) \cap S|. \]
If we let  $\mathcal{T}_{\bm{z}} := \Psi^{-1}(\{\bm{z}\}) \cap \mathcal{T}$ for convenience, we then have 
\[  |\Psi(\mathcal{T})| = \sum_{\bm{z} \in \Psi(\mathcal{T})} 1 = \sum_{\bm{z} \in \Psi(\mathcal{T})} \frac{|\mathcal{T}_{\bm{z}}|}{|\mathcal{T}_{\bm{z}}|}
\geq \sum_{\bm{z} \in \Psi(\mathcal{T})} \frac{|\mathcal{T}_{\bm{z}}|}{m(\Psi;\mathcal{T})} = \frac{|\mathcal{T}|}{m(\Psi;\mathcal{T})}, \]
so that we always have the lower bound
\begin{equation}\label{eq:multiplicity_bound_m}
|\mathcal{T}| \leq m(\Psi;\mathcal{T}) |E|.
\end{equation}
\begin{remark}\label{remark_m_interpretation}
Observe that for $\Psi$ as given by \eqref{eq:map_Psi} the quantity $m(\Psi;\mathcal{T})$ is the maximum number of solutions $(n_1, m_1, \ldots, n_k, m_k) \in \mathcal{T}$ to the system of diophantine equations given by 
\[ \gamma(n_1) + \ldots + \gamma(n_k) = (\bm{z} - \bm{y}) + \gamma(m_1) + \ldots + \gamma(m_k), \]
as $\bm{z}$ ranges over $\Psi(\mathcal{T})$ -- and by construction of $\mathcal{T}$, $\bm{z} \neq \bm{y}$. We then see that we are dealing with an inhomogeneous version of \eqref{eq:generic_diophantine_system}, and therefore $m(\Psi;\mathcal{T})$ should be compared with the quantity $J_{k,\gamma}(N)$ as defined in \eqref{eq:conjectured_number_of_solutions}. 
\end{remark}
Combining \eqref{eq:multiplicity_bound_m} and \eqref{eq:lowerbound_tower} one has therefore 
\[ \alpha^k \beta^k \lesssim m(\Psi;\mathcal{T}) |E|, \] 
which since $\alpha \leq N$ implies
\[ \alpha^{k+1} \beta^k \lesssim N\, m(\Psi;\mathcal{T}) |E|. \]
Comparing this with the combinatorial reformulation \eqref{eq:refinement_inequality_subcritical_p_q} one sees that to conclude an optimal subcritical estimate on the $q = p'$ line it is sufficient to show that we can construct a tower $\mathcal{T}$ (obtained by flowing $k$ times in each direction) such that $m(\Psi;\mathcal{T}) \lesssim_\epsilon N^{k-1 + \epsilon}$ can be shown to hold. This is precisely the strategy that we adopt in the proof of Theorem \ref{main_theorem}.
\begin{remark}
Comparing once again the quantities $m(\Psi;\mathcal{T})$ with $J_{k,\gamma}$, we stress the fact that we are looking for an estimate of the form $m(\Psi;\mathcal{T}) \lesssim_\epsilon N^{k-1 + \epsilon}$, whereas, for the corresponding homogeneous system \eqref{eq:generic_diophantine_system}, estimate \eqref{eq:conjectured_number_of_solutions} in this regime takes instead the form $J_{k,\gamma}(N) \lesssim_{\epsilon} N^{k + \epsilon}$ (and this bound clearly cannot be improved because of the presence of diagonal solutions to \eqref{eq:generic_diophantine_system}).
\end{remark}

\begin{proof}[Proof of Theorem \ref{main_theorem}]
By \eqref{eq:refinement_inequality_subcritical_p_q} of Section \ref{section_reformulation} the inequalities \eqref{eq:3/2->3_estimate},\eqref{eq:5/3->5/2_estimate},\eqref{eq:7/4->7/3_estimate} of Theorem \ref{main_theorem} can be reformulated as, respectively,
\begin{align*}
\alpha^2 \beta^{} &\lesssim_{\epsilon} N^{1+\epsilon} |E|, \\
\alpha^3 \beta^2 &\lesssim_{\epsilon} N^{2+\epsilon} |E|, \\
\alpha^4 \beta^3 &\lesssim_{\epsilon} N^{3+\epsilon} |E|.
\end{align*}
We claim that we can always assume that $\alpha, \beta \gg 1$. Indeed, if $\alpha \lesssim 1$ or $\beta \lesssim 1$ we have $\alpha^k \beta^k \lesssim N^k$ for any $k$; but since we see easily\footnote{Ultimately a consequence of the fact that $\gamma$ is essentially injective.} that $\alpha \lesssim |E|$, all the desired inequalities would immediately follow.\par
Assuming then $\alpha, \beta \gg 1$, as anticipated above we proceed to prove the \emph{sharpened}\footnote{Notice indeed that the inequalities as written are false in general, without the additional $\alpha, \beta \gg 1$ assumption.} inequalities
\begin{align*}
\alpha^{} \beta^{} &\lesssim_{\epsilon} N^{\epsilon} |E|, \\
\alpha^2 \beta^2 &\lesssim_{\epsilon} N^{1+\epsilon} |E|, \\
\alpha^3 \beta^3 &\lesssim_{\epsilon} N^{2+\epsilon} |E|,
\end{align*}
from which the desired ones follow immediately since $\alpha \leq N$.\par
We prove these inequalities all at once, conditionally on Lemmata \ref{lemma_number_of_solutions_case_1}, \ref{lemma_number_of_solutions_case_2}, \ref{lemma_number_of_solutions_case_3} which are proven in Section \ref{section_number_of_solutions}. Let $k\in \{1,2,3\}$ and recall that by Section \ref{section_preliminaries_reductions} we can assume $\gamma(n) = P(n)$ with $\deg P \geq 2$ when $k=1$, $\gamma(n) = (n, P(n))$ when $k=2$ and $\gamma(n)=(n,n^2,n^3)$ when $k=3$. For each $k$ build the tower $\mathcal{T}$ as per Section \ref{subsection_tower_construction} by flowing $k$ times in each direction, so that by \eqref{eq:lowerbound_tower} and \eqref{eq:multiplicity_bound_m} we have 
\[ \alpha^k \beta^k \lesssim m(\Psi;\mathcal{T}) |E|, \]
with $\Psi$ given by \eqref{eq:map_Psi}. By construction the tower $\mathcal{T}$ is contained in the set 
\[ \{(n_1,m_1,\ldots, n_k,m_k) \in [1,N]^{2k} : \gamma(n_1) - \gamma(m_1) + \ldots + \gamma(n_k) - \gamma(m_k) \neq 0\}\]
(see Remark \ref{remark_pruned_tower_property}) and therefore $m(\Psi;\mathcal{T})$ is bounded by the maximum number of solutions in $[1,N]^{2k}$ to 
\[ \gamma(n_1) - \gamma(m_1) + \ldots + \gamma(n_k) - \gamma(m_k) = \mathfrak{z} \]
when $\mathfrak{z} \neq 0$ (see Remark \ref{remark_m_interpretation}). By substituting for $\gamma$ the respective special forms for each $k$ we see that
\begin{itemize}
\item when $k=1$, $m(\Psi;\mathcal{T}) \lesssim_\epsilon N^{\epsilon}$ by Lemma \ref{lemma_number_of_solutions_case_1} of Section \ref{section_number_of_solutions};
\item when $k=2$, $m(\Psi;\mathcal{T}) \lesssim_\epsilon N^{1 + \epsilon}$ by Lemma \ref{lemma_number_of_solutions_case_2} of Section \ref{section_number_of_solutions};
\item when $k=3$, $m(\Psi;\mathcal{T}) \lesssim_\epsilon N^{2 + \epsilon}$ by Lemma \ref{lemma_number_of_solutions_case_3} of Section \ref{section_number_of_solutions}.
\end{itemize}
The proof is thus concluded, modulo the proofs of the lemmata which are presented in the next section.
\end{proof}
\section{Bounds for the number of solutions to diophantine systems of equations}\label{section_number_of_solutions}
In this last section we conclude the proof of Theorem \ref{main_theorem} by proving the lemmata for the number of solutions of the relative diophantine equations employed above. Such lemmata are proven by elementary means, ultimately resting on the divisor bound; the proofs are inspired by the corresponding one in \cite{HughesWooley} by the second author and Wooley.\par
The lemmata are ordered by increasing number of equations and increasing number of variables. The proof of the first one already contains in essence the idea for all three proofs.
\begin{lemma}\label{lemma_number_of_solutions_case_1}
For every $\epsilon>0$ the following holds.\\
Let $P \in \mathbb{Z}[X]$ with $\deg P \geq 2$. The number of solutions $(n_1,m_1) \in [1,N]^2$ to 
\begin{equation}
P(n_1) - P(m_1) = \mathfrak{z}_1  \label{eq:diophantine_equation_case_1}
\end{equation}
with $|\mathfrak{z}_1|\lesssim N^{\deg P}$ and $\mathfrak{z}_1 \neq 0$ is bounded by $\lesssim_{\epsilon} N^{\epsilon}$.
\end{lemma}
\begin{proof}
Observe that since $P$ is not linear it must be that for some non-vanishing non-constant $Q \in \mathbb{Z}[X,Y]$ we have  identically
\[ P(X) - P(Y) = Q(X,Y)(X - Y); \]
moreover, for any $n$ the univariate polynomial $Q(n,Y)$ is non-constant. But the polynomial identity implies that any solution to \eqref{eq:diophantine_equation_case_1} is also a solution to one of the systems of diophantine equations
\begin{equation*}
\begin{cases}
Q(n_1,m_1) &= d_1, \\
n_1 - m_1 &= d_2,
\end{cases}
\end{equation*}
with $d_1 d_2 = \mathfrak{z}_1$. Since there are only $\lesssim_\epsilon N^{\epsilon}$ such factorisations of $\mathfrak{z}_1 \neq 0$ (by the divisor bound) and since each distinct such system has clearly at most $\lesssim_{\deg P} 1$ solutions (since $Q(n_1,Y)$ is non-constant), we conclude that \eqref{eq:diophantine_equation_case_1} has at most $\lesssim_{\epsilon,\gamma} N^{\epsilon}$ solutions.
\end{proof}
The second lemma deals with a system of two diophantine equations in four variables in which the first equation is linear.
\begin{lemma}\label{lemma_number_of_solutions_case_2}
For every $\epsilon>0$ the following holds.\\
Let $P \in \mathbb{Z}[X]$ with $\deg P \geq 2$. The number of solutions $(n_1, m_1, n_2, m_2) \in [1,N]^4$ to 
\begin{equation*}
\begin{cases}
n_1 - m_1 + n_2 - m_2 &= \mathfrak{z}_1, \\
P(n_1) - P(m_1) + P(n_2) - P(m_2) &= \mathfrak{z}_2,
\end{cases} 
\end{equation*}
with $|\mathfrak{z}_1| \lesssim N, |\mathfrak{z}_2|\lesssim N^{\deg P}$ and $\mathfrak{z}_1,\mathfrak{z}_2$ not both simultaneously zero is bounded by $\lesssim_{\epsilon} N^{1 + \epsilon}$.
\end{lemma}
The proof of Lemma \ref{lemma_number_of_solutions_case_2} rests on the fact that if $P$ is non-linear then there is a non-vanishing polynomial $Q \in \mathbb{Z}[X,Y,Z]$ such that 
\begin{equation}
P(X) - P(Y) + P(Z) - P(X-Y+Z) = Q(X,Y,Z)(X-Y)(Y-Z) \label{eq:polynomial_identity}
\end{equation}
identically (and moreover, for any $n,m$ the univariate polynomial $Q(n,m,Z)$ does not vanish identically). The resulting proof is essentially a simpler version of the proof of the next lemma, and therefore we omit the details and direct the reader there.\par
The final lemma deals with a special system of three diophantine equations in six variables -- effectively an inhomogeneous version of one of the so-called Vinogradov systems.
\begin{lemma}\label{lemma_number_of_solutions_case_3}
For every $\epsilon > 0$, the following holds.\\
The number of solutions $(n_1, m_1, n_2, m_2, n_3, m_3) \in [1,N]^6$ to the diophantine system of equations 
\begin{equation}
\begin{cases}
n_1 - m_1 + n_2 & = \mathfrak{z}_1 + m_2 - n_3 + m_3, \\
n_1^2 - m_1^2 + n_2^2 & = \mathfrak{z}_2 + m_2^2 - n_3^2 + m_3^2, \\
n_1^3 - m_1^3 + n_2^3 & = \mathfrak{z}_3 + m_2^3 - n_3^3 + m_3^3, 
\end{cases} \label{eq:diophantine_system_moment_curve_3+3_variables}
\end{equation}
with $|\mathfrak{z}_1| \lesssim N, |\mathfrak{z}_2| \lesssim N^2, |\mathfrak{z}_3| \lesssim N^3$ and $\mathfrak{z}_1, \mathfrak{z}_2, \mathfrak{z}_3$ not all simultaneously zero is bounded by $\lesssim_\epsilon N^{2+\epsilon}$.
\end{lemma}
\begin{proof}
For solutions of the type we want, the quantity $u := n_1 - m_1 + n_2$ can take at most $\lesssim N$ values. Fix then such a value $u \lesssim N$ and observe that using the polynomial identities (particular cases of \eqref{eq:polynomial_identity})
\begin{align*}
X^2 - Y^2 + Z^2 - (X - Y + Z)^2 &= 2(X-Y)(Y-Z), \\
X^3 - Y^3 + Z^3 - (X - Y + Z)^3 &= 3(X+Z)(X-Y)(Y-Z),
\end{align*}
we can rewrite system \eqref{eq:diophantine_system_moment_curve_3+3_variables} as 
\begin{equation}
\begin{cases}
n_1 - m_1 + n_2  &= u,\\
m_2 - n_3 + m_3 & = u - \mathfrak{z}_1, \\
2(n_1 - m_1)(m_1 - n_2) & = 2(m_2 - n_3)(n_3 - m_3) \\
& \hspace{1em} - u^2 + \mathfrak{z}_2 + (u-\mathfrak{z}_1)^2, \\
3(n_1 + n_2)(n_1 - m_1)(m_1 - n_2) & = 3(m_2+m_3)(m_2 - n_3)(n_3 - m_3) \\
 & \hspace{1em} - u^3 +\mathfrak{z}_3 + (u-\mathfrak{z}_1)^3.
\end{cases}  \label{eq:factorised_diophantine_system}
\end{equation}
We stress that if $(n_1, m_1, n_2, m_2, n_3, m_3) \in [-N,N]^6$ is a solution to \eqref{eq:diophantine_system_moment_curve_3+3_variables} then for some value of $u\lesssim N$ it is a solution to \eqref{eq:factorised_diophantine_system} too.\par 
Multiplying by $2$ and using the quadratic equation, we see that we can rewrite the cubic equation above as 
\begin{align*}
3(n_1 + n_2)2&(n_1 - m_1)(m_1 - n_2) \\
&= 3(m_2+m_3) [2(n_1 - m_1)(m_1 - n_2) + u^2 - \mathfrak{z}_2 - (u-\mathfrak{z}_1)^2] \\
& \hspace{1em} +2(\mathfrak{z}_3 + (u-\mathfrak{z}_1)^3 - u^3),
\end{align*}
and rearranging 
\begin{equation*}
\begin{aligned}
6(n_1 + n_2 - & m_2 - m_3)(n_1 - m_1)(m_1 - n_2) \\
& = 3(m_2+m_3)[u^2 - \mathfrak{z}_2 - (u-\mathfrak{z}_1)^2] +2(\mathfrak{z}_3 + (u-\mathfrak{z}_1)^3 - u^3).
\end{aligned} 
\end{equation*}
Letting $t = m_2 + m_3$ we observe that, once $t$ is also fixed, if 
\[ M := 3t[u^2 - \mathfrak{z}_2 - (u-\mathfrak{z}_1)^2] +2(\mathfrak{z}_3 + (u-\mathfrak{z}_1)^3 - u^3) \neq 0 \]
then since $M \lesssim N^3$ this number can be factorised as $M = 6 \,d_1 d_2 d_3$ into at most $\lesssim_\epsilon N^\epsilon$ ways, by the divisor bound. Solutions to \eqref{eq:diophantine_system_moment_curve_3+3_variables} are therefore solutions to one of the systems 
\begin{equation*}
\begin{cases}
n_1 - m_1 + n_2  &= u,\\
m_2 - n_3 + m_3 & = u - \mathfrak{z}_1, \\
m_2 + m_3 &= t, \\
n_1 + n_2 - m_2 - m_3 &= d_1, \\
n_1 - m_1 &= d_2, \\
m_1 - n_2 &= d_3,
\end{cases}
\end{equation*}
obtained by choosing $u,t$ and factorising $M$. This system is not linearly independent (it has rank 5) but it contains enough equations to fix the values of, say, $n_1,m_1,n_2$. By the quadratic equation of \eqref{eq:factorised_diophantine_system} we have then
\begin{equation}
2(m_2 - n_3)(n_3 - m_3) = 2 d_2 d_3 + u^2 - \mathfrak{z}_2 - (u-\mathfrak{z}_1)^2.  \label{eq:quadratic_equation_intermediate_step}
\end{equation}
If the right-hand side of \eqref{eq:quadratic_equation_intermediate_step} is non-zero we can invoke the divisor bound again to factor it as $2d_4 d_5$ and thus reduce to the system 
\begin{equation*}
\begin{cases}
m_2 - n_3 + m_3 & = u - \mathfrak{z}_1, \\
m_2 + m_3 &= t, \\
m_2 - n_3 &= d_4, \\
n_3 - m_3 &= d_5,
\end{cases}
\end{equation*}
which can have at most $1$ solution (if any). If the right-hand side of \eqref{eq:quadratic_equation_intermediate_step} is instead zero then it must be either $m_2 - n_3 =0$ or $n_3 - m_3 = 0$; either of them gives a well-posed system of linear equations and thus at most $1$ solution. This analysis has thus shown that for values of $u,t$ such that $M \neq 0$ the system has at most $\lesssim_\epsilon N^{2\epsilon}$ solutions; since $|u|,|t| \lesssim N$, we obtain from these situations a contribution of at most $\lesssim_\epsilon N^{2+2\epsilon}$ solutions to \eqref{eq:diophantine_system_moment_curve_3+3_variables}.\par
The case in which $M = 0$ for some choice of parameters has to be dealt with separately, that which involves the gruelling analysis of several subcases (the assumption that $\mathfrak{z}_1,\mathfrak{z}_2,\mathfrak{z}_3$ are not all simultaneously zero is crucial here). However, to show that the cases in which $M=0$ contribute at most some more $\lesssim_\epsilon N^{2+\epsilon}$ solutions to \eqref{eq:diophantine_system_moment_curve_3+3_variables} altogether, the very same arguments used above (or variations thereof) suffice; hence we omit the details. The proof of the Lemma is thus concluded.
\end{proof}

\bibliographystyle{abbrv}
\bibliography{bibliography}

\end{document}